\documentclass[12pt,leqno,a4paper]{amsart}
\usepackage{amssymb,enumerate}

\usepackage{txfonts}
\usepackage{mathrsfs}
\usepackage{hyperref}
\usepackage[all]{xy}

\newcommand{\FF}{\mathbb{F}}
\newcommand{\ZZ}{\mathbb{Z}}

\newcommand{\bC}{\mathbf{C}}
\newcommand{\bG}{\mathbf{G}}
\newcommand{\bL}{\mathbf{L}}
\newcommand{\bc}{\mathbf{c}}
\newcommand{\bm}{\mathbf{m}}

\newcommand{\cB}{\mathcal{B}}
\newcommand{\cE}{\mathcal{E}}
\newcommand{\cF}{\mathcal{F}}
\newcommand{\cU}{\mathcal{U}}
\newcommand{\cW}{\mathcal{W}}

\newcommand{\fS}{\mathfrak{S}}

\newcommand{\scC}{\mathscr{C}}
\newcommand{\scT}{\mathscr{T}}
\newcommand{\scU}{\mathscr{U}}

\newcommand{\Aut}{\operatorname{Aut}\nolimits}
\newcommand{\IBr}{\operatorname{IBr}\nolimits}
\newcommand{\Ind}{\operatorname{Ind}}
\newcommand{\Lin}{\operatorname{Lin}}
\newcommand{\Irr}{\operatorname{Irr}\nolimits}
\newcommand{\Out}{\operatorname{Out}\nolimits}
\newcommand{\Res}{\operatorname{Res}}
\newcommand{\GL}{\operatorname{GL}}
\newcommand{\GU}{\operatorname{GU}}
\newcommand{\Sp}{\operatorname{Sp}}
\newcommand{\GO}{\operatorname{GO}}
\newcommand{\SO}{\operatorname{SO}}
\newcommand{\CSp}{\operatorname{CSp}}
\newcommand{\PSp}{\operatorname{PSp}}

\newcommand{\J}{\operatorname{J}}

\newcommand{\tpsi}{\tilde{\psi}}
\newcommand{\ts}{\tilde{s}}
\newcommand{\wB}{\widetilde{B}}
\newcommand{\tG}{\widetilde{G}}
\newcommand{\tX}{\widetilde{X}}
\newcommand{\wR}{\widetilde{R}}
\newcommand{\wcB}{\widetilde{\cB}}
\newcommand{\tbG}{\widetilde{\bG}}
\newcommand{\wchi}{\widetilde{\chi}}
\newcommand{\trho}{\tilde{\rho}}

\newcommand{\wOm}{\widetilde{\Omega}}
\newcommand{\tw}[1]{{}^{#1}\!}

\let\eps=\epsilon
\let\ga=\gamma
\let\veps=\varepsilon
\let\la=\lambda
\let\vhi=\varphi
\let\Ga=\Gamma
\let\ka=\kappa

\theoremstyle{theorem}
\newtheorem{mainthm}{Theorem}
\newtheorem{thm}{Theorem}[section]
\newtheorem{lem}[thm]{Lemma}
\newtheorem{prop}[thm]{Proposition}
\newtheorem{cor}[thm]{Corollary}

\theoremstyle{definition}

\newtheorem{rmk}[thm]{Remark}

\newtheorem*{rem}{Remark}
\numberwithin{equation}{section}

\begin{document}

\title[Alperin weight condition for type $\mathsf C$]{The inductive blockwise Alperin weight condition\\ for type $\mathsf C$ and the prime $2$}

\author{Zhicheng Feng}
\address{School of Mathematics and Physics, University of Science and Technology Beijing, Beijing 100083, China}
\email{zfeng@pku.edu.cn}

\author{Gunter Malle}
\address{FB Mathematik, TU Kaiserslautern, Postfach 3049, 67653 Kaiserslautern,
  Germany}
\email{malle@mathematik.uni-kl.de}

\thanks{The first author gratefully acknowledges financial support by NSFC~(11631001 and 11901028) and Fundamental Research Funds for the Central Universities (No. FRF-TP-19-036A1). The second author gratefully acknowledges financial support by SFB TRR 195.}

\begin{abstract}
We establish the inductive blockwise Alperin weight condition for simple groups
of Lie type $\mathsf C$ and the bad prime $2$.  As a main step, we derive a
labelling set for the irreducible $2$-Brauer characters of the finite
symplectic groups $\Sp_{2n}(q)$ (with odd $q$), together with the action of
automorphisms. As a further important ingredient we prove a Jordan
decomposition for weights.
\end{abstract}



\date{\today}

\maketitle


\section{Introduction}

Let $G$ be a finite group and $\ell$ a prime. An \emph{$\ell$-weight of $G$} is
a pair $(R,\vhi)$, where $R$ is an $\ell$-subgroup of $G$ and
$\vhi\in\Irr(N_G(R)/R)$ is of defect zero. We also view $\vhi$ as a
character of $N_G(R)$. The group $G$ acts by conjugation on the set of weights.
Let $B$ be an $\ell$-block of $G$. A weight $(R,\vhi)$ of $G$ is a
\emph{$B$-weight} if $b^G=B$, where $b$ is the block of $N_G(R)$ containing
$\vhi$ and $b^G$ is the induced block. We denote by $\cW(B)$ the set of
$G$-conjugacy classes of $B$-weights.
In 1986, Jonathan Alperin \cite{Al87} stated his famous conjecture:
$$|{\IBr(B)}|=|\cW(B)|\qquad\text{for every block $B$ of every finite
  group $G$}.$$

In 2011, Navarro--Tiep \cite{NT11} proved a reduction theorem for the
blockfree version of Alperin's weight conjecture, and soon after, Sp\"ath
\cite{Sp13} reduced the above blockwise version to the \emph{inductive
blockwise Alperin weight condition} for simple groups (see \S~\ref{sec:iBAW for
sym}). If that condition is satisfied for all finite simple groups, then the
Alperin weight conjecture holds for all blocks of all finite groups. So far
this inductive condition has been verified for several families of simple
groups: groups of Lie type in their defining characteristic \cite{Sp13},
alternating groups, Suzuki and Ree groups \cite{Ma14}, groups of types $G_2$
and $\tw3D_4$ \cite{Sch16}, and certain cases for classical groups
\cite{Feng19,FLZ19a,Li19,LZ18}.

The most difficult case for groups of Lie type $G$ seems to be when $\ell$ is a
bad prime for $G$; so far the results for classical groups are restricted to
good primes only. In this paper, we consider simple groups of Lie type
$\mathsf C$ and the (only) bad prime $\ell=2$.
Our main result is the following.

\begin{mainthm}\label{mainthm}
 The inductive blockwise Alperin weight condition holds for the finite simple
 groups $\PSp_{2n}(q)$, $q$ odd, and the prime $2$.
\end{mainthm}

The crucial step to establish Theorem \ref{mainthm} is to construct a blockwise
$\Aut(G)$-equivariant bijection between $\IBr(G)$ and $\cW(G)$ for
$G=\Sp_{2n}(q)$, which can be divided into two main parts: determine the action
of $\Aut(G)$ on $\IBr(G)$ and on $\cW(G)$. By constructing an equivariant Jordan
decomposition for both the irreducible $2$-Brauer characters (using the deep
result of Bonnaf\'e--Rouquier \cite{BR03}) and the $2$-weights, we reduce this
problem to the unipotent $2$-block.

For the Brauer characters, Chaneb \cite{Ch18} proved that there is a
unitriangular basic set for the unipotent $2$-block of $G$. Using the
construction given by Geck--H\'ezard \cite{GH08} and Taylor \cite{Tay13}, we
show that there is an $\Aut(G)$-equivariant bijection between the irreducible
Brauer characters in the unipotent $2$-block of $G$ and the unipotent classes
of $G$. From this and the knowledge of the action of automorphisms on ordinary
characters \cite{CS17} we obtain a labelling set for irreducible Brauer
characters, together with the action of automorphisms. On the other hand, the
$2$-weights of $G$ have been classified by An \cite{An93}. Using the list of
$2$-weights there and the methods in \cite{FLZ19a}, we determine the action of
$\Aut(G)$ on $\cW(G)$ from the parametrisations.
\medskip

This paper is structured as follows. After introducing some notation
in \S  \ref{sec:Preli} we formulate a criterion for the inductive blockwise
Alperin weight condition for $\PSp_{2n}(q)$ and the prime $2$ in
\S \ref{sec:iBAW for sym}. The action of automorphisms on the irreducible
$2$-Brauer characters is considered in \S \ref{sec:action-Brauer}, while their
action on the $2$-weights is determined in \S \ref{sec:act wei}. From this we
derive a Jordan decomposition for weights, which we show to hold for a large
class of classical groups. Finally, we establish an equivariant bijection
between the irreducible $2$-Brauer characters and the conjugacy classes of
$2$-weights in \S \ref{sec:equ bijection uni}, thus completing the proof of
our Main Theorem~\ref{mainthm}.
\vskip 1pc
\noindent{\bf Acknowledgement:} We thank Jay Taylor for comments on an
earlier version.

\section{Preliminaries}   \label{sec:Preli}

\subsection{General notation}

Let $G$ be a finite group. Concerning the block and character theory of $G$ we
mainly follow the notation of \cite{Na98}, where for induction and restriction
we use the notation $\Ind$ and $\Res$. For a normal subgroup $N\unlhd G$ we
sometimes identify the (Brauer) characters of $G/N$ with the (Brauer)
characters of $G$ whose kernel contains $N$. 

If a group $A$ acts on a finite set $X$, we denote by $A_x$ the stabiliser of
$x\in X$ in $A$, analogously we denote by $A_{X'}$ the setwise stabiliser of
$X'\subseteq X$.
Let $\ell$ be a prime. If $A$ acts on a finite group $G$ by automorphisms, then
there is a natural action of $A$ on $\Irr(G)\cup\IBr_\ell(G)$ given by
${}^{a^{-1}}\chi(g)=\chi^a(g)=\chi(g^{a^{-1}})$ for every $g\in G$, $a\in A$
and $\chi\in\Irr(G)\cup\IBr_\ell(G)$. For $P\le G$ and
$\chi\in\Irr(G)\cup\IBr_\ell(G)$, we denote by $A_{P,\chi}$ the stabiliser of
$\chi$ in $A_P$. If $A$ is abelian, we denote by $\Lin_{\ell'}(A)$ the complex
linear characters of $A$ of $\ell'$-order and we also identify
$\Lin_{\ell'}(A)$ with $\IBr(A)$.

Let $G\unlhd \tG$ be finite groups with abelian factor group $\tG/G$. For a
weight $(\wR,\widetilde\vhi)$ of $\tG$, Brough and Sp\"ath \cite{BS19} defined
the weights of $G$ covered by $(\wR,\widetilde\vhi)$. In addition, they also
give a Clifford theory for weights between $G$ and $\tG$.

\subsection{Reductive groups}
Let $q$ be a power of a prime $p$. We let $\overline\FF_q$ be an algebraic
closure of the finite field $\FF_q$. As usual, algebraic groups are denoted by
boldface letters. Suppose that $\bG$ is a connected reductive algebraic group
over $\overline\FF_q$ and $F:\bG\to\bG$ a Frobenius endomorphism endowing $\bG$
with an $\FF_q$-structure. Let $\bG^*$ be Langlands dual to $\bG$ with
corresponding Frobenius endomorphism also denoted $F$.

We refer to \cite{GM20} for notation pertaining to the character theory of
finite reductive groups. Let $\ell$ be a prime number different from $p$.
For a semisimple $\ell'$-element $s$ of ${\bG^*}^F$, we denote by
$\cE_\ell(\bG^F,s)$ the union of the Lusztig series $\cE(\bG^F,st)$, where $t$
runs through semisimple $\ell$-elements of ${\bG^*}^F$ commuting with $s$.
By a theorem of Brou\'e--Michel (see \cite[Thm.~9.12]{CE04}), the set
$\cE_\ell(\bG^F,s)$ is a union of $\ell$-blocks of $\bG^F$.

\subsection{Some notation and conventions for symplectic groups} \label{notations-and-conventions}

Let $p$ be an odd prime, and $q$ a power of $p$.
We follow the notation from \cite{An93,FS89}.

We recall from \cite{FS89} that there exists a set $\cF$ of polynomials serving
as elementary divisors for all semisimple elements of symplectic  and orthogonal
groups.
We denote by  $\Irr(\FF_q[x])$ the set of all monic irreducible polynomials
over $\FF_q$. For each $\Delta$ in $\Irr(\FF_q[x])\setminus \{x\}$, we define
$\Delta^*\in\Irr(\FF_q[x])$ to be the polynomial whose roots are the inverses
of the roots of $\Delta$. Now, we denote by
\begin{align*}
 \cF_0&=\left\{ x-1,x+1 ~\right\},\\
 \cF_1&=\left\{ \Delta\in\Irr(\FF_q[x])\mid \Delta\notin \cF_0,\,
   \Delta\neq x,\,\Delta=\Delta^* ~\right\},\\
 \cF_2&=\left\{~ \Delta\Delta^* \mid \Delta\in\Irr(\FF_q[x])\setminus \cF_0,\,
   \Delta\neq x,\,\Delta\ne\Delta^* ~\right\}.
\end{align*}
Let $\cF=\cF_0\cup\cF_1\cup\cF_2$.
For $\Ga\in\cF$ denote by $d_\Ga$ its degree and by $\delta_\Ga$ its
\emph{reduced degree} defined by
$$\delta_\Ga=
  \begin{cases} d_\Ga & \text{if}\ \Ga\in\cF_0, \\
     \frac{1}{2}d_\Ga & \text{if}\ \Ga\in\cF_1\cup \cF_2.
\end{cases}$$
Since the polynomials in $\cF_1\cup \cF_2$ have even degree, $\delta_\Ga$ is
an integer. In addition, we mention a sign $\veps_\Ga$ for
$\Ga\in\cF$ defined by
$$\veps_\Ga=
  \begin{cases} \varepsilon & \text{if}\ \Ga\in\cF_0, \\
    -1 & \text{if}\ \Ga\in\cF_1, \\
     1 & \text{if}\ \Ga\in\cF_2.
\end{cases}$$

Let $V$ be a finite-dimensional symplectic or orthogonal space over $\FF_q$.
We denote by $I(V)$ the group of isometries of $V$. Any semisimple element
$s\in I(V)$ induces a unique orthogonal decomposition
$$
  V=\sum\limits_{\Ga\in\cF} V_\Ga(s), \quad s=\prod_{\Ga\in\cF}s(\Ga),
$$
where the $V_\Ga(s)$ are non-degenerate subspaces of $V$,
$s(\Ga)\in I(V_\Ga(s))$, and $s(\Ga)$ has minimal polynomial $\Ga$.  Let
$m_\Ga(s)$ be the multiplicity of $\Ga$ in $s(\Ga)$. If $m_\Ga(s)\ne 0$, then
we say $\Ga$ is an \emph{elementary divisor} of $s$. The centraliser of $s$ in
$I(V)$ has a decomposition $C_{I(V)}(s)=\prod_{\Ga}C_\Ga(s)$, where, by
\cite[(1.13)]{FS89}, 
\begin{equation}   \label{eq:centr}
  C_\Ga(s)=C_{I(V_\Ga(s))}(s(\Ga))=
    \begin{cases} I(V_\Ga(s)) & \text{if}\ \Ga\in\cF_0, \\
    \GL_{m_\Ga(s)}(\veps_\Ga q^{\delta_\Ga}) & \text{if}\ \Ga\in\cF_1\cup\cF_2.
\end{cases}
\addtocounter{thm}{1}\tag{\thethm}
\end{equation}
Here, as customary, $\GL_m(-q)$ means $\GU_m(q)$.

\section{Symplectic groups and the inductive blockwise Alperin weight condition}
\label{sec:iBAW for sym}

Throughout we let $G=\Sp_{2n}(q)$, $\tG=\CSp_{2n}(q)$ for an odd prime power
$q$ and we consider modular representations with respect to the prime~$2$. 
We view $G$ and $\tG$ as the groups of fixed points under Frobenius
endomorphisms of certain connected reductive algebraic groups: Let
$\bG=\Sp_{2n}(\overline\FF_q)$, $\tbG=\CSp_{2n}(\overline\FF_q)$. Note that
$\tbG$ has connected centre, so the containment $\bG\le\tbG$ is a regular
embedding (see \cite[Def.~1.7.1]{GM20}). Then $G=\bG^F$ and $\tG=\tbG^F$,
where $F$ is the standard Frobenius endomorphism. In addition, $\tbG^*$ is the
corresponding special Clifford group and $\bG^* = \SO_{2n+1}(\overline\FF_q)$.
We write $\pi:\tbG^*\to\bG^*$ for the surjection induced by the regular
embedding $\bG\hookrightarrow\tbG$. Here, $\tG^* = (\tbG^*)^F$ is a special
Clifford group over $\FF_q$ and $G^* = {\bG^*}^F=\SO_{2n+1}(q)$.

\begin{lem}   \label{block-typeC}
 \begin{enumerate}[\rm(a)]
  \item For every semisimple $2'$-element $\ts\in\tG^*$, $\cE_2(\tG,\ts)$ is
   a $2$-block of $\tG$
  \item For every semisimple $2'$-element $s\in G^*$, $\cE_2(G,s)$ is a
   $2$-block of $G$.
 \end{enumerate}
\end{lem}

\begin{proof}
This follows by \cite[Thm.~21.14]{CE04}. 
\end{proof}

We write $\tilde B_{\ts}:=\cE_2(\tG,\ts)$ and $B_s:=\cE_2(G, s)$ for the
2-blocks of $\tG$ and $G$ parametrised by semisimple $2'$-elements
$\ts\in\tG^*$ and $s\in G^*$ respectively.

Let $F_p$ be the field automorphism of $G$ which sends $(a_{ij})$ to
$(a_{ij}^p)$ and let $\delta$ be a diagonal automorphism of $G$ (\emph{i.e.},
the conjugation by an element of $\tG\setminus GZ(\tG)$). Then by
\cite[Thm.~2.5.1]{GLS98}, $\tG\rtimes\langle F_p \rangle$ induces all
automorphisms of $G$.
In addition, $\Out(G)=\langle\delta\rangle\times\langle F_p \rangle$.

\begin{cor}   \label{block-covering}
 Let $B$ be a $2$-block of $G$. Then
 \begin{enumerate}[\rm(a)]
  \item $B^\delta=B$.
  \item There are $(q-1)_{2'}$ blocks of $\tG$ covering $B$. Furthermore, if
   $B=B_s$ and $\ts\in\pi^{-1}(s)$, then the blocks of $\tG$ covering $B$ are
   $\wB_{\tilde z\ts}$, where $\tilde z$ runs through
   $\mathcal O_{2'}(Z(\tG^*))$.
 \end{enumerate}
\end{cor}

\begin{proof}
Part~(a) is a direct consequence of Lemma \ref{block-typeC} as diagonal
automorphism stabilise all semisimple conjugacy classes and hence all
Lusztig series, and (b) follows from the fact that $\cE_2(\tG,\ts)$ contains
constituents of the induction to $\tG$ of the characters in $\cE_2(G,s)$, and
for $z_1,z_2\in Z(\tG^*)$, $\tilde z_1\ts$ and $\tilde z_2\ts$ are not
$\tG^*$-conjugate if $z_1\ne z_2$
(cf. \cite[(2D)]{FS89}).
\end{proof}

For a finite group $H$, we denote by $\cW(H)$ the set of $H$-conjugacy classes
of weights of $H$. We write $\overline{(R,\vhi)}\in\cW(H)$ for the class of a
weight $(R,\vhi)$ of $H$. If $\cB$ is a union of blocks of $H$, we write
$\cW(\cB)=\bigcup\limits_{B\in \cB}^\cdot \cW(B)$.

We use the criterion for the inductive blockwise Alperin weight condition given
by Brough and Sp\"ath \cite[Thm.]{BS19}.  As we will be considering groups of
type $\mathsf C$, we also make some simplifying assumptions: cyclic quotient
$\tX/X$ and cyclic $D$.

\begin{thm}[{\cite[Thm.]{BS19}}]   \label{thm:criterion}
 Let $S$ be a finite non-abelian simple group and $\ell$ a prime dividing $|S|$.
 Let $X$ be the full $\ell'$-covering group of $S$, $B$ an $\ell$-block of $X$
 and assume there are groups $\tX$, $D$ such that $X \unlhd \tX \rtimes D$ and
 the following hold:
 \begin{enumerate}[\rm(1)]\setlength{\itemsep}{0pt}
		\item \begin{enumerate}[\rm(i)]\setlength{\itemsep}{0pt}
			\item $X=[\tX,\tX]$, 
			\item $C_{\tX D}(G)=Z(\tX)$ and $\tX D/Z(\tX) \cong \Aut(X)$,
			\item $\Out(X)$ is abelian,
			\item both $\tX/X$ and $D$ are cyclic.
		\end{enumerate}
		\item Let $\wcB$ be the union of $\ell$-blocks of $\tX$ covering $B$.
		There exists a $\Lin_{\ell' }(\tX/X) \rtimes D_{\wcB}$-equivariant
		bijection $\wOm_{\wcB}: \IBr(\wcB) \to \cW(\wcB)$ such that
		$\wOm_{\wcB}(\IBr(\tilde B)) = \cW(\tilde B)$
		for every $\tilde B \in \wcB$,
		and $\J_X(\tpsi) = \J_X(\wOm_{\wcB}(\tpsi))$ for every $\tpsi \in \IBr({\wcB})$.
		\item For every $\tpsi\in\IBr(\wcB)$, there exists some $\psi_0\in\IBr(X\mid\tpsi)$ such that
		$(\tX\rtimes D)_{\psi_0}=\tX_{\psi_0}\rtimes D_{\psi_0}$,
		\item In every $\tX$-orbit on $\cW(B)$, there is a $(R,\vhi_0)$ such
		that $(\tX D)_{R,\vhi_0} = \tX_{R,\vhi_0} (XD)_{R,\vhi_0}$,
 \end{enumerate}
 Then the inductive blockwise Alperin weight condition holds for the block $B$.
\end{thm}

For the definitions of $\J_X(\tpsi)$ and $\J_X(\wOm_{\wcB}(\tpsi))$, see
\cite[\S 2]{BS19}.

Considering the simple group $\PSp_{2n}(q)$ and the prime $\ell=2$, we have the
following.

\begin{prop}   \label{ibaw-condition-C-2}
 Let $S=\PSp_{2n}(q)$ (with $n\ge2$ and odd $q$) and $G=\Sp_{2n}(q)$. Then the
 inductive blockwise Alperin weight condition holds for $S$ and the prime $2$
 if the following are satisfied:
 \begin{enumerate}[\rm(1)]
  \item There is an $\Aut(G)$-equivariant bijection $\Omega:\IBr(G)\to\cW(G)$
   preserving blocks;
  \item for any $\chi\in\IBr(G)$ and  $\sigma\in\langle F_p\rangle$, if
   $\chi^\delta=\chi^\sigma$ then
   $\chi^\delta=\chi^\sigma=\chi$; and
  \item for any 2-weight $(R,\vhi)$ of $G$ and  $\sigma\in\langle F_p\rangle$,
 if
   $\overline{(R,\vhi)}^\delta=\overline{(R,\vhi)}^{\sigma}$ then
   $\overline{(R,\vhi)}^\delta=\overline{(R,\vhi)}^{\sigma}=\overline{(R,\vhi)}$.
 \end{enumerate}
\end{prop}

\begin{proof}
By \cite[\S 6.1]{GLS98}, the group $G$ is the universal covering group of $S$.
Let $X=S$, $\tilde X=\mathrm{PCSp}_{2n}(q)$ and $D=\langle F_p\rangle$.
Then Theorem~\ref{thm:criterion}~(1) holds.
When considering Theorem~\ref{thm:criterion}~(2)--(4), we may replace $X$,
$\tilde X$ by $G$, $\tG=\CSp_{2n}(q)$ respectively, since $Z(G)$ is contained
in the kernel of every irreducible 2-Brauer character of $G$ and in the kernel
of every weight character. Therefore, the conditions~(3) and (4) of
Theorem~\ref{thm:criterion} are guaranteed by~(2) and (3) respectively.

For a $2$-block $B=B_s$ of $G$, let $\ts\in\pi^{-1}(s)$ and $\wB=\wB_{\ts}$.
Let $\wcB$ be the union of blocks of $\tG$ covering $B$. Then by
Corollary~\ref{block-covering}, $\wcB$ is the union of blocks
$\wB_{\tilde z\ts}$, where $\tilde z$ runs through $\mathcal O_{2'}(Z(\tG^*))$.
In addition, $D_{\wcB}=D_B$. 

Let $\chi\in\Irr(B)$. Note that $|\IBr(\tG\mid\chi)|\le (q-1)_{2'}$. So by
Corollary~\ref{block-covering}~(b), $|\IBr(\wB)\cap \IBr(\tG\mid\chi)|=1$ and
we write $\wchi$ for the unique character in $\IBr(\wB)\cap \IBr(\tG\mid\chi)$.
On the other hand, we let $\overline{(\wR,\widetilde\vhi)}\in\cW(\wB)$ such
that $(\wR,\widetilde\vhi)$ covers $\Omega(\chi)$. Since $\Omega$ is
$\Aut(G)$-equivariant, $\overline{(\wR,\widetilde\vhi)}$ is also unique by a
similar argument. We then define
$\wOm_{\wB}(\wchi):=\overline{(\wR,\widetilde\vhi)}$.
If $\chi$ runs through a $\langle\delta\rangle$-transversal in
$\IBr(B)$, then $\wchi$ runs through $\IBr(\wB)$. Thus, $\wOm_{\wB}$ is a
bijection from $\IBr(\wB)$ to $\cW(\wB)$ since $\Omega$ is
$\Aut(G)$-equivariant. We extend $\wOm_{\wB}$ to $\wcB$:
for $\eta\in \Lin_{\ell'}(\tG/G)$ we set $\wOm_{\wcB}(\eta \wchi):=
\overline{(\wR,\Res^{\tG}_{N_{\tG}(\wR)}(\eta)\widetilde\vhi)}$.
Thus $\wOm_{\wcB}$ is a $\Lin_{\ell'}(\tG/G)$-equivariant bijection from
$\IBr(\wcB)$ to $\cW(\wcB)$.

We claim that $\wOm_{\wcB}$ is $D_{\wcB}$-equivariant. In fact, for
$\wchi\in\IBr(\wB)$, $\chi\in\IBr(G\mid \wchi)$ and $\sigma\in D_{\wcB}$, the
character $\wchi^\sigma$ is the unique irreducible character in $\wB^\sigma$
lying over $\chi^\sigma$. Let
$\overline{(\wR,\widetilde\vhi)}:=\wOm_{\wcB}(\wchi)$. Then
$\overline{(\wR,\widetilde\vhi)}^\sigma$ is the unique weight in $\wB^\sigma$
covering $\Omega(\chi)^\sigma=\Omega(\chi^\sigma)$. By the definition of
$\wOm_{\wcB}$, one have
$\wOm_{\wcB}(\wchi^\sigma)=\overline{(\wR,\widetilde\vhi)}^\sigma$. This implies
that $\wOm_{\wcB}$ is $D_{\wcB}$-equivariant and preserves blocks.

Finally, the condition $\J_G(\tpsi) = \J_G(\wOm_{\wcB}(\tpsi))$ for
$\psi\in\IBr(\wcB)$ can be verified from the assumption that $\Omega$ is
$\Aut(G)$-equivariant. Indeed, since $\tG/G$ is cyclic, $\J_G(\tpsi)$ (or
$\J_G(\wOm_{\wcB}(\tpsi))$) is determined by the $\ell$-part of the number of
Brauer characters (of weights) which are $\tG$-conjugate to $\tpsi$ (or
$\wOm_{\wcB}(\tpsi)$). Thus condition~(2) of Theorem~\ref{thm:criterion} holds
and this completes the proof.
\end{proof}

\begin{rmk}   \label{equiva-condition(ii)(iii)}
We note that in Proposition \ref{ibaw-condition-C-2}, if (1) holds, then (2) and
(3) are equivalent.
\end{rmk}

\section{Action of automorphisms on $2$-Brauer characters}
\label{sec:action-Brauer}

\subsection{Action of automorphisms on $2$-Brauer characters in the unipotent $2$-block of $\Sp_{2n}(q)$}

In this section we relate the irreducible Brauer characters in the principal
2-block of $G=\Sp_{2n}(q)$ to the unipotent conjugacy classes of $G$. The main
tool is a result of Chaneb which in turn relies on results of Geck--H\'ezard
and Taylor.

We start by recalling the parametrisation of unipotent conjugacy classes and the
action of outer automorphisms. Let $q$ be a power of an odd prime $p$. The
unipotent conjugacy classes of $\bG=\Sp_{2n}(\overline\FF_q)$ are parametrised
by their Jordan normal forms. There exist unipotent elements with Jordan block
shape described by a partition $\la\vdash 2n$ in $\bG$ if and only if any odd
part occurs an even number of times in $\la$. That is, if $c_i$ denotes the
number of parts of
$\la$ equal to $i$, then $ic_i\equiv0\pmod2$ for all $i$. It follows from this
that all unipotent classes have representatives already over the prime field
$\FF_p$ and thus are $F$-stable for any Frobenius endomorphism $F$ of $\bG$.
We write $\bC_\la$ for the unipotent class of $\bG$ parametrised by $\la$.
Furthermore, the group of components $A_\bG(u):=C_\bG(x)/C_\bG^\circ(x)$ for
$x\in \bC_\la$ is $(\ZZ/2\ZZ)^{a(\la)}$ where $a(\la)$ is the number of even $i$
with $c_i>0$ (see e.g. \cite[\S I.2]{Sp82}). Thus $\bC_\la^F\cap\bG^F$ splits
into $2^{a(\la)}$ classes of $\bG^F$, which we denote $C_{\la,i}$,
$1\le i\le 2^{a(\la)}$.   \par
For $\la\vdash 2n$ let us set $\delta_\la=1$ if there is some even part in
$\la$ with odd multiplicity, and $\delta_\la=0$ otherwise. The unipotent
classes of $\tbG$ are in bijection with those in $\bG$, but the group of
components for $\bC_\la$ has half the size in $\tbG$ if, and only if,
$\delta_\la=1$ (see e.g. \cite[\S13.1]{Ca}). Let $F:\tbG\to\tbG$ be a Frobenius
endomorphism. Then $\tbG^F$ induces the outer diagonal automorphism on $\bG^F$.
It follows that the outer diagonal automorphism of $\bG^F$ fixes exactly those
unipotent classes of $\bG^F$ with $\delta_\la=0$ and interchanges the others in
pairs. Moreover, it can be shown that all unipotent $\bG^F$-classes are
preserved by any field automorphism (see e.g. \cite[Prop.~3.3]{CS17}):

\begin{prop}   \label{prop:aut uni}
 Let $C_{\la,i}\subset G$ be a unipotent conjugacy class. Then
 \begin{enumerate}
  \item[\rm(a)] $C_{\la,i}$ is invariant under any field automorphism; and
  \item[\rm(b)] $C_{\la,i}$ is $\tG$-invariant if and only if $\delta_\la=0$.
 \end{enumerate}
\end{prop}

Following H\'ezard \cite{Hez}, Taylor \cite[Prop.~3.11]{Tay12} describes a map
$\Psi:\cU(\bG)\to\Irr(G)$ from the set $\cU(\bG)$ of unipotent classes of $\bG$
to $\Irr(G)$ with the following properties: let $\bC_\la\in\cU(\bG)$ be a
unipotent class. Then $\Psi(\bC_\la)$ lies in the Lusztig series of a
quasi-isolated 2-element $\ts_\la\in\tG^*$ with centraliser of type $D_aB_b$,
where $a+b=n$, whose image under Lusztig's Jordan decomposition is a special
unipotent character of $\SO_{2a}^+(q)\SO_{2b+1}(q)$ parametrised by a
non-degenerate symbol. Furthermore, if $\cF_\la$ denotes the Lusztig family in
$\cE(\SO_{2a}^+(q)\SO_{2b+1}(q),1)$ containing the Jordan correspondent of
$\Psi(\bC_\la)$, then the associated finite group (see \cite[4.2.15]{GM20}) is
exactly $A_{\tbG}(u)$, where $u\in \bC_\la$. Here, $\ts_\la=1$ (that is, $a=0$)
if and only if $\delta_\la=0$. 

According to \cite[Prop.~4.3]{GH08} for each class $\bC_\la$ there exist
$|A_{\tbG}(u)|$ characters $\trho_{\la,i}\in\cE(\tG,\ts_\la)$, with Jordan
correspondents in the family $\cF_\la$, such that the matrix of multiplicities
of their Alvis--Curtis--Kawanaka--Lusztig duals in the generalised
Gelfand--Graev characters corresponding to $\bC_\la$ is the identity matrix.

Further, Taylor \cite[Prop.~5.4]{Tay13} shows that for each class $\bC_\la$
there exist $|A_{\bG}(u)|$ characters $\rho_{\la,i}\in\cE(G,s_\la)$, where
$s_\la=\pi(\ts_\la)$, lying below the $\trho_{\la,i}$, such that again the
matrix of multiplicities of their duals in the GGGRs corresponding to $\bC_\la$
is the identity matrix. Furthermore, the construction in the proof of
\cite[Prop.~5.4]{Tay13} shows that these sets are invariant under the outer
diagonal automorphism of $G$. Let
$$\Xi(G):=
  \big\{\rho_{\la,i}\mid \bC_\la\in\cU(\bG),\ 1\le i\le2^{a(\la)}\,\big\}$$
and let $B_1$ be the unipotent (principal) $2$-block of $G=\Sp_{2n}(q)$.

\begin{prop}   \label{prop:IBR-uni-class}
 There exists an $\Aut(G)$-equivariant bijection between $\IBr(B_1)$ and
 $\cU(G)$.
\end{prop}

\begin{proof}
We shall show this in two steps, with $\Xi(G)$ as intermediary. Firstly,
according to Chaneb \cite[Thm.~2.8]{Ch18} the set $\Xi(G)$ is a basic set for
$B_1$ such that the decomposition matrix with respect to it is uni-triangular.
Now the diagonal automorphisms stabilise all rational Lusztig series (see
\cite[Prop.~2.6.17]{GM20}) and field
automorphisms stabilise all Lusztig series for quasi-isolated 2-elements of
$G^*$ (for example by the parametrisation of these classes). Furthermore,
all unipotent characters of groups of type $D_a$ parametrised by
non-degenerate symbols, and all unipotent characters of groups of type $B_b$
are fixed by all automorphisms (see \cite[Thm.~4.5.11]{GM20}), thus $\Xi(G)$ is
$\Aut(G)$-invariant. Hence, the uni-triangular decomposition matrix induces an
$\Aut(G)$-equivariant bijection $\Xi(G)\to\IBr(B_1)$.
\par
On the other hand, as discussed above, via \cite[Prop.~5.4]{Tay13} we obtain a
bijection between $\Xi(G)$ and the set $\cU(G)$ of unipotent classes of $G$.
By Proposition~\ref{prop:aut uni}, field automorphisms act trivially on
$\cU(G)$ and the diagonal outer automorphism only moves those classes in
$\bC_\la$ with $\delta_\la=1$. Now note that for such $\la$, the centraliser
$C_\bG(s)$ is disconnected (see \cite[Lemma~3.7]{Tay12}) and thus the
restriction of any character in $\cE(\tG,\ts)$ to $G$ has two constituents in
$\cE(G,s)$. Since $\Xi(G)$ is stable under diagonal automorphisms, this shows
that none of the $\rho_{\la,i}$, with $\delta_\la=1$, is $\tG$-invariant. Thus
any bijection between $\Xi(G)$ and $\cU(G)$ sending characters $\rho_{\la,i}$
to classes in $\bC_\la\cap G$ is $\Aut(G)$-equivariant. The claim follows.
\end{proof}

We reformulate our combinatorial results on the unipotent Brauer characters
of $G$ as follows: Let $\scU(n)$ be the set of maps
$\bm:\ZZ_{\ge 1}\to \ZZ_{\ge 0}$ such that $\sum_{j\ge 1} j\bm(j)=2n$ and
$\bm(j)$ is even if $j$ is odd. Here, $\bm$ is counted $2^{k_{\bm}}$-times,
where $k_{\bm}=|\{ j\mid j\text{ is even and }\bm(j)\ne 0 \}|$.
Then $\scU(n)$ is a labelling set for $\cU(G)$ and hence for $\IBr(B_1)$.
Let $\scU_1(n)$ be the subset of $\scU(n)$ consisting of those $\bm$ such that
$\bm(j)$ is even for all $j$. Again, an element $\bm\in\scU_1(n)$ is counted
$2^{k_{\bm}}$-times in $\scU_1(n)$. Then by Propositions~\ref{prop:aut uni} and
\ref{prop:IBR-uni-class}, we immediately have the following.

\begin{cor}   \label{cor:aut IBR}
 Let  $\chi\in\IBr(B_1)$ correspond to $\bm\in \scU(n)$. Then
 \begin{enumerate}
  \item[\rm(a)] $\chi$ is invariant under any field automorphism; and
  \item[\rm(b)] $\chi$ is $\tG$-invariant if and only if $\bm\in \scU_1(n)$.
 \end{enumerate}
\end{cor}

\subsection{A Jordan decomposition for Brauer characters}

We give an equivariant Jordan decomposition for Brauer characters of
$G=\Sp_{2n}(q)$. Let $s$ be a semisimple $2'$-element of $G^*$. Then
\begin{equation}\label{decom-L}
  C_{G^*}(s)^*\cong \Sp_{m_{x-1}(s)-1}(q)\times \prod\limits_{\Ga\in\cF_1\cup \cF_2} \GL_{m_\Ga(s)}(\veps_\Ga q^{\delta_\Ga})
\addtocounter{thm}{1}\tag{\thethm}
\end{equation}
(see (\ref{eq:centr})). By \cite[Thm.~21.14]{CE04}, $C_{G^*}(s)^*$ has a unique
unipotent $2$-block, the principal $2$-block.

For $\sigma=F_p$, let $\sigma^*$ be the automorphism of $G^*$ dual to $\sigma$
as in \cite[\S5.3]{Tay18}. Then $\sigma^*$ is also the field automorphism of a
special orthogonal group which sends $(a_{ij})$ to $(a_{ij}^p)$.

By a result of Bonnaf\'e--Rouquier \cite{BR03}, there is a Morita equivalence
between the unipotent $2$-block of $C_{G^*}(s)^*$ and $B_s$, from which we
obtain:

\begin{prop}   \label{Jor-dec-Bar}
 Let $s$ be a semisimple $2'$-element of $G^*$ and $b_s$ the unipotent
 $2$-block of $C_{G^*}(s)^*$. Then there is a bijection $\mathfrak J_s$ between 
 $\IBr(b_s)$ and $\IBr(B_s)$ such that for $\sigma\in\langle F_p\rangle$ or
 $\sigma=\delta$ the following diagram commutes
	$$\xymatrix{
		\IBr(b_s) \ar[r]^{\sigma}\ar[d]_{\mathfrak J_s}& \IBr(b_{s_\sigma}) \ar[d]_{\mathfrak J_{s_\sigma}}\\
		\IBr(B_s)\ar[r]^{\sigma}& \IBr(B_{s_\sigma})}$$
 where $s_\sigma={\sigma^*}^{-1}(s)$ if $\sigma\in\langle F_p\rangle$, and
 $s_\sigma=s$ if $\sigma=\delta$. 
\end{prop}

Note that if $\sigma\in\langle F_p\rangle$, then $\sigma$ induces an isomorphism
between $C_{G^*}(s)^*$ and $C_{G^*}({\sigma^*}^{-1}(s))^*$, and for
$\sigma=\delta$, by composition of $\delta$ with an inner automorphism, we
assume that $\delta$ induces an automorphism of $C_{G^*}(s)^*$.

\begin{proof}
Since $s$ has odd order, $C_{\bG^*}(s)$ is an $F$-stable Levi subgroup of
$\bG^*$. We let $\bL$ be an $F$-stable Levi subgroup of $\bG$ in duality with
$C_{\bG^*}(s)$ and $L=\bL^F$. By \cite{BR03} the map
$\Psi_s:\Irr(b_s)\to\Irr(B_s)$, $\la\mapsto R^{\bG}_{\bL}(\hat s \la)$, induces
a Morita equivalence between $b_s$ and $B_s$. Here $\hat s$ is a linear
character of $L$ defined as in \cite[Prop.~2.5.20]{GM20}. Thus, if
$\Xi\subseteq \Irr(b_s)$ is a unitriangular basic set for $b_s$, then
$\Psi_s(\Xi)$ is a unitriangular basic set for $B_s$ and the respective
decomposition matrices coincide. 

According to~(\ref{decom-L}) we write $L=L_0\times L_+$, where $L_0$ is a
symplectic group and $L_+$ is a product of general linear and unitary groups.
Note that $b_s=b_0\otimes b_+$ with $b_0$ the principal 2-block of $L_0$ and
$b_+$ the principal 2-block of $L_+$, and thus
$\Irr(b_s)=\Irr(b_0)\times\Irr(b_+)$ and similarly for the Brauer characters.
We now choose $\Xi=\Xi_0\times \Xi_+$ as follows. We take for $\Xi_0$ the basic
set $\Xi(L_0)$ in the proof of Proposition~\ref{prop:IBR-uni-class}. Then
$\Xi_0$ is $\Aut(L_0)$-invariant. Further, since the unipotent characters form
a unitriangular basic set for unipotent 2-blocks of general unitary and linear
groups (cf. \cite{Ge91}), we choose $\Xi_+$ to be the set of unipotent
characters of $L_+$. By \cite[Thm.~4.5.11]{GM20} this is invariant under
$\Aut(L_+)$. Thus for any $\sigma\in\Aut(G)$, if $L^\sigma=L$, then $\Xi_+$
(and then $\Xi$) is $\langle\sigma\rangle$-stable. Thus there is an
$\Aut(G)_L$-equivariant bijection $\upsilon_s$ between $\Xi$ and $\IBr(b_s)$.

If $\sigma\in\langle F_p\rangle$, by \cite[Prop.~9.2]{Tay18} we get $\Psi_s(\la)^\sigma
= R^{\bG}_{\bL}(\hat s^\sigma\la^\sigma)=\Psi_{{\sigma^*}^{-1}(s)}(\la^\sigma)$
since $\hat s^\sigma=\widehat{{\sigma^*}^{-1}(s)}$ by \cite[Prop.~7.2]{Tay18}.

Now let $\sigma=\delta$, normalised as in the statement. Let
$\la=\la_0\times\la_+\in\Xi$ with $\la_i\in\Xi_i$ for
$i\in\{0,+\}$. Let  $\mu_{-1}$ be the Lusztig symbol of the unipotent character
of $C_{L_0^*}(t_0)$, which corresponds to $\la_0$ under Jordan decomposition,
corresponding to the elementary divisor~$x+1$.
Hence by \cite[Lemma~5.3]{FLZ19a}, we have $\la^\delta\ne\la$ if and only if
$-1$ is  an eigenvalue of $t_0$ and $\mu_{-1}$ is non-degenerate. By
\cite[Prop.~3.3.20]{GM20} we have $\Psi_s(\la)\in\cE(G,st)$, where, up to
conjugacy, $t=t_0\times 1_{L_+^*}$. Then $C_{G^*}(st)=C_{L_0^*}(t_0)$ and the
unipotent character of $C_{L_0^*}(t)$ corresponding to $\la$ under the Jordan
decomposition and the unipotent character of $C_{G^*}(st)$ corresponding to
$\Psi_s(\la)$ under the Jordan decomposition coincide. So by
\cite[Lemma~5.3]{FLZ19a} again, we have $\Psi_s(\la)^\delta\ne\Psi_s(\la)$ if
and only if $-1$ is  an
eigenvalue of $t_0$ and $\mu_{-1}$ is non-degenerate, and if and only if 
$\la^\delta\ne\la$. Similarly we can show $\Psi_s(\la)^\delta=\Psi_s(\la')$ if
and only if $\la^\delta=\la'$.
 
Thus the union of all such sets $\Psi_s(\Xi)$, where $s$ runs through a complete
set of representatives of $G^*$-conjugacy classes of semisimple $2'$-element of
$G^*$, is an $\Aut(G)$-invariant unitriangular basic set for $G$, and thus
there is an $\Aut(G)$-equivariant  bijection between this basic set and
$\IBr(G)$. Then with $\Upsilon_s$ the restriction of this bijection between
$\Psi_s(\Xi)$ and $\IBr(B_s)$ we take
$\mathfrak J_s=\Upsilon_s\circ\Psi_s\circ\upsilon^{-1}_s$. So when considering
the action of automorphisms, we may deal with $\Xi$, $\Psi_s(\Xi)$ instead of
$\IBr(b_s)$, $\IBr(B_s)$ respectively. Therefore our assertions follow from
the properties of $\Psi_s$ obtained in the above paragraph.
\end{proof}

By the above arguments, $G$ has an $\Aut(G)$-stable unitriangular basic set.
Then by \cite[Thm.~3.1]{CS17} we have the following result.

\begin{cor}   \label{cor:(3)}
 Let $\chi\in\IBr(G)$ and $\sigma\in\langle F_p \rangle$. If
 $\chi^\delta=\chi^\sigma$, then $\chi^\delta=\chi^\sigma=\chi$.
\end{cor}

In this way, Proposition~\ref{ibaw-condition-C-2}~(2) holds for $G$ and then by
Remark~\ref{equiva-condition(ii)(iii)}, in order to verify the inductive
blockwise Alperin weight condition for $\PSp_{2n}(q)$ and the prime $2$, it
remains to construct an $\Aut(G)$-equivariant bijection between $\IBr(G)$ and
$\cW(G)$ which preserves $2$-blocks.

\section{Action of automorphisms on $2$-weights}
\label{sec:act wei}

For a non-negative integer $m$, we denote by $\mathscr P(m)$ the set of
partitions $\la\vdash m$, and $\scT(m)$ denotes the set of triples
$(\la_1,\la_2,\kappa)$, where $\la_1$ and $\la_2$ are partitions, and $\kappa$
is a $2$-core such that $|\la_1|+|\la_2|+|\kappa|=m$.
A partition $\la$ is determined uniquely by its $2$-core $\kappa'$ and
$2$-quotient $(\la_1,\la_2)$ (cf.~\cite[\S 3]{Ol93}).
We define the set $\scT'(m)$ consisting of tuples
$(\la_1,\la_2,\la_3,\kappa_1,\kappa_2)$, where $\la_1$,
$\la_2$ and $\la_3$ are partitions, $\kappa_1$ and $\kappa_2$ are
$2$-cores such that $|\la_1|+2(|\la_2|+|\la_3|)+|\kappa_1|+|\kappa_2|=m$.
Then there is canonical bijection between $\scT(m)$ and $\scT'(m)$:
$(\la_1,\la_2,\kappa)\in\scT(m)$ corresponds to
$(\la_1,\la_2',\la_3',\kappa_1',\kappa)\in \scT'(m)$ such that $\la_2$ has
$2$-core $\kappa_1'$ and $2$-quotient $(\la_2',\la_3')$.

\subsection{Action of automorphisms on $2$-weights of $\Sp_{2n}(q)$}

Let $2^{a+1}$ be the exact power of $2$ dividing $q^2-1$, so that $a\ge 2$.
Let $\veps$ be the sign chosen so that $2^a$ divides $q-\veps$.
Given $\Ga\in\cF$, let $\alpha_\Ga$ and $m_\Ga$ be the integers such
that $2^{\alpha_\Ga}$ is the exact power of $2$ dividing $\delta_\Ga$ and
$m_\Ga 2^{\alpha_\Ga}=\delta_\Ga$.

We need to recall An's description of 2-weights of $\Sp_{2n}(q)$ \cite{An93}.
Let $\Ga\in\cF$ be a polynomial whose roots have $2'$-order. For an integer
$\ga\ge 0$ and a sequence $\bc=(c_1,c_2,\ldots,c_t)$ of non-negative integers,
we let $R_{\Ga,\ga,\bc}$ and $V_{\Ga,d}$ (with $d=\ga+c_1+c_2+\cdots+c_t$) be
defined as in \cite[p.~190]{An93}.
Recall that $\dim(V_{\Ga,d})=2^{d+1}\delta_\Ga$ and
$$R_{\Ga,\ga,\bc}=
\begin{cases} 
R_{m_\Ga,\alpha_\Ga,\ga}\wr A_\bc & \text{if}\ \veps_\Ga=\veps\ \text{and}\ \Ga \ne x-1, \\
S_{m_\Ga,\alpha_\Ga,\ga-1}\wr A_\bc & \text{if}\ \veps_\Ga=-\veps\ \text{and}\ \ga \ge 1, \\
\langle -I_{d_\Ga} \rangle\wr A_\bc & \text{if}\ \veps_\Ga=-\veps,\ga=0\ \text{and}\ c_1\ne 1,\\
E_{d_\Ga,1,1}\wr A_{\bc'}& \text{if}\ \veps_\Ga=-\veps,\ga=0\ \text{and}\ c_1= 1,\\
Q_{1,0,\ga}\wr A_\bc & \text{if}\ \Ga=x-1 \ \text{and}\ a=2,\\
Q_{1,0,\ga}\wr A_\bc & \text{if}\ \Ga=x-1, a>2 , \\
E_{1,\alpha,1}\wr A_\bc & \text{if}\ \Ga=x-1,a>2 \ \text{and}\ \ga=0,
\end{cases}$$
where $R_{m_\Ga,\alpha_\Ga,\ga}$, $S_{m_\Ga,\alpha_\Ga,\ga}$,
$E_{d_\Ga,1,1}$,
$Q_{1,0,\ga}$,
$E_{1,\alpha,1}$, $A_\bc$ are defined as in \cite[\S 1 and \S 2]{An93},
$\bc'=(c_2,\ldots,c_t)$, $E_{d_\Ga,1,1}\cong E_{1,\alpha,1}\cong Q_8$ and
$\alpha=0$ or $1$.
Then $R_{\Ga,\ga,\bc}$ is determined uniquely up to conjugacy in $G_{\Ga,d}=\Sp(V_{\Ga,d})$ by $\Ga$, $\ga$ and $\bc$ except when  $\Ga=x-1$, $a>2$ and $\ga=0$.
In this case, there are three conjugacy classes of 
$R_{\Ga,\ga,\bc}$ in $\Sp(V_{\Ga,d})$:
$Q_{1,0,0}\wr A_\bc$,
$E_{1,0,1}\wr A_\bc$ and $E_{1,1,1}\wr A_\bc$.

Let $C_{\Ga,\ga,\bc}=C_{G_{\Ga,d}}(R_{\Ga,\ga,\bc})$ and
$N_{\Ga,\ga,\bc}=N_{G_{\Ga,d}}(R_{\Ga,\ga,\bc})$.
Let $\theta_{\Ga,\ga,\bc}$ (also sometimes denoted $\theta_{\Ga,d}$) be the
character $\theta_\Ga\otimes I$ of $C_{\Ga,\ga,\bc}$, where $\theta_\Ga$ is
defined as in \cite[\S 4]{An93}. Let $\scC_{\Ga,d}$ be the set of characters of
$(N_{\Ga,\ga,\bc})_{\theta_{\Ga,d}}$ which are of $2$-defect zero when
viewed as characters of $(N_{\Ga,\ga,\bc})_{\theta_{\Ga,d}}/R_{\Ga,\ga,\bc}$,
where $\ga$ and $\bc$ satisfy that $\ga+c_1+\cdots+c_t=d$.
Then by \cite[(6C)]{An93}, $|\scC_{\Ga,d}|=2^d$ if $\Ga\in\cF_1\cup\cF_2$,
$|\scC_{x-1,d}|=2^{d+1}$ if $d\ge 1$, and $|\scC_{x-1,d}|=3$ if $d=0$.

\begin{lem}   \label{act-wei-basic-case}
 Let $\sigma$ be an automorphism of $G_{\Ga,d}$
 such that $\theta_{\Ga,d}^{\sigma}=\theta_{\Ga,d}$. Then
 \begin{enumerate}[\rm(a)]
  \item if $\sigma\in\langle F_p\rangle$, then $\sigma$ fixes every element of
   $\scC_{\Ga,d}$, and
  \item if $\sigma=\delta$, then 
   \begin{itemize}
    \item[(1)] $\sigma$ fixes every element of $\scC_{\Ga,d}$ if  
     $\Ga\in\cF_1\cup\cF_2$,
	\item[(2)] if $d\ge 1$, then there are $2^d$ characters of $\scC_{x-1,d}$
     which are invariant under the action of $\sigma$ and the others can be
     partitioned into $2^{d-1}$ pairs and $\sigma$ transposes each pair, and
    \item[(3)] if $d=0$, then one character of $\scC_{x-1,d}$ is fixed by
     $\sigma$, and the other two characters are swapped by $\sigma$.
  \end{itemize}
 \end{enumerate}
\end{lem}

\begin{proof}
Let $\psi$ be a character of $(N_{\Ga,\ga,\bc})_{\theta_{\Ga,d}}$ in
$\scC_{\Ga,d}$. By \cite[(4G) and p.~195]{An93}, $\psi$ is determined uniquely
by $\theta_{\Ga,d}$ (and then by $\Ga$) if $\Ga\ne x-1$ or $\ga\ge 1$. If
$\Ga=x-1$ and $\ga=0$, we first assume that $\bc=(0)$, and then
$\dim(V_{\Ga,d})=2$. The action of automorphisms on $\scC_{x-1,0}$ was given in
\cite[\S 5]{LL19}: $F_p$ fixes every element of $\scC_{x-1,0}$ and $\delta$
fixes one element of $\scC_{x-1,0}$ and exchanges the other two. Precisely,
if $a=2$, then there are three characters of $\scC_{x-1,0}$ associated with
$Q_{1,0,0}$ and one is $\langle\delta\rangle$-invariant and the other two are
exchanged by $\delta$. The three characters of $\scC_{x-1,0}$ are associated to
$Q_{1,0,0}$, $E_{1,0,1}$, $E_{1,1,1}$ respectively and the $G$-conjugacy class
of $Q_{1,0,0}$ is $\langle\delta\rangle$-stable while the $G$-conjugacy classes
of $E_{1,0,1}$ and $E_{1,1,1}$ are exchanged by $\delta$.
This shows (a), (b)(1) and (b)(3).

Now let $\Ga=x-1$, $d\ge 1$ and $\sigma=\delta$. Then by the proof of
\cite[(6C)]{An93}, $\scC_{x-1,d}$ contains $2^{d-1}$ characters of
$(N_{\Ga,\ga,\bc})_{\theta_{\Ga,d}}$ when $\ga\ge 1$ and $3\cdot 2^{d-1}$
characters of $(N_{\Ga,\ga,\bc})_{\theta_{\Ga,d}}$ when $\ga=0$. By the above
paragraph, the $2^{d-1}$ characters of $(N_{\Ga,\ga,\bc})_{\theta_{\Ga,d}}$
with $\ga\ge 1$ are fixed by $\sigma$ and there are $2^{d-1}$ characters of
$(N_{\Ga,0,\bc})_{\theta_{\Ga,d}}$ which are also  fixed by $\sigma$.
The other $2^{d}$ characters of $(N_{\Ga,0,\bc})_{\theta_{\Ga,d}}$ can be
partitioned into $2^{d-1}$ pairs and $\sigma$ transposes each pair.
This completes the proof.
\end{proof}

We write $\scC_{\Ga,d}=\{ \psi_{\Ga,d,i} \mid 1\le i\le |\scC_{\Ga,d}|\}$.
If $\Ga=x-1$ and $d\ge 1$, then we also write
$\scC_{x-1,d}=\{ \psi_{x-1,d,i,j} \mid 1\le i\le 2, 1\le j\le 2^d\}$ and by
Lemma~\ref{act-wei-basic-case}, we may assume that
\begin{equation}\label{conven-weights-basic-case-1}
\begin{split}
  &\psi_{x-1,d,1,j}^\delta=\psi_{x-1,d,1,j}\ \textrm{if $1\le j\le 2^d$, and}\\
  &\psi_{x-1,d,2,j}^\delta=\psi_{x-1,d,2,j+2^{d-1}}\ \textrm{if $1\le j\le 2^{d-1}$}.
\end{split} \addtocounter{thm}{1}\tag{\thethm}
\end{equation}
If $\Ga=x-1$ and $d=0$, then we write
$\scC_{x-1,0}=\{ \psi_{x-1,0,i,1} \mid 1\le i\le 3\}$ and choose notation so
that
\begin{equation}\label{conven-weights-basic-case-2}
  \psi_{x-1,0,1,1}^\delta=\psi_{x-1,0,1,1} \ \text{and}\ 
  \psi_{x-1,0,2,1}^\delta=\psi_{x-1,0,3,1}.
  \addtocounter{thm}{1}\tag{\thethm}
\end{equation}
We also sometimes write $\theta_{\Ga,d,i}$ for the character $\theta_{\Ga,d}$
lying below $\psi_{\Ga,d,i}$.
\medskip

Let $B=B_s$ where $s\in G^*$ is a semisimple $2'$-element. Define $w_\Ga(s)$ to
be the integer satisfying $m_\Ga(s)=w_\Ga(s)$ or $m_{\Ga}(s)=2w_\Ga(s)+1$
according as $\Ga\in \cF_1\cup\cF_2$ or $\Ga=x-1$. We also abbreviate $w_\Ga$
for $w_\Ga(s)$ when no confusion can arise.

\begin{prop}[An (1993)]
 The set
 $\scT(w_{x-1})\times\prod\limits_{\Ga\in \cF_1\cup\cF_2} \mathscr P(w_\Ga)$ is
 a labelling set for $\cW(B)$.
\end{prop}

We recall the main steps of the construction of this labelling from the proof
of \cite[(6D)]{An93}:

Let $(R,\vhi)$ be a $B$-weight of $G$, $C=C_G(R)$, $N=N_G(R)$, and
$\theta\in\Irr(C\mid \vhi)$. Let $\psi$ be the irreducible character of
$N_\theta$ covering $\theta$ such that
$\vhi=\Ind^N_{N_\theta}\psi$. Then $R$ decomposes as \cite[(3A)]{An93}
$$V=V_1\perp \cdots  \perp V_u\perp V_{u+1}\perp\cdots\perp V_t,$$
$$R=R_1\times \cdots \times R_u\times R_{u+1}\times\cdots\times R_t,$$
where  $\dim V_j\ge 2$ for $j\ge 1$, $R_k=\{\pm I_{V_k}\}$ for $1\le k\le u$,
and $R_i$ are basic subgroups of $\Sp(V_i)$ for $i\ge u+1$.
Let $G_i=\Sp(V_i)$, $C_i=C_{G_i}(R_i)$ and $N_i=N_{G_i}(R_i)$.
Then $C=\prod\limits_{i=1}^t C_i$. 
Let  $\theta=\prod\limits_{i=1}^t \theta_i$, where $\theta_i$ is an
irreducible character of $C_i R_i$ trivial on $R_i$ for $i\ge 1$.

Now we rewrite the decompositions
$\theta=\prod\limits_{\Ga,d,i} \theta_{\Ga,d,i}^{t_{\Ga,d,i}}$, $R=\prod\limits_{\Ga,d,i}R_{\Ga,d,i}^{t_{\Ga,d,i}}.$
Then $m_{\Ga}(s)=\prod\limits_{d,i} t_{\Ga,d,i}\beta_\Ga 2^d$ for each $\Ga$. Here $\beta_\Ga=1$ or $2$ according as $\Ga\in\cF_1\cup\cF_2$ or $\Ga=x-1$.
Then we have
$$N_\theta=\prod_{\Ga,d,i}
 (N_{\Ga,d,i})_{\theta_{\Ga,d,i}}\wr\fS(t_{\Ga,d,i}),
\quad  \psi=\prod_{\Ga,d,i}  \psi_{\Ga,d,i}$$
with $\psi_{\Ga,d,i}$ a character of
$(N_{\Ga,d,i})_{\theta_{\Ga,d,i}}\wr\fS(t_{\Ga,d,i})$ covering $\theta_{\Ga,d,i}^{t_{\Ga,d,i}}$ and of $2$-defect zero as a character of $\left( (N_{\Ga,d,i})_{\theta_{\Ga,d,i}}\wr\fS(t_{\Ga,d,i})\right)/
R_{\Ga,d,i}^{t_{\Ga,d,i}}$.
By Clifford theory, $\psi_{\Ga,d,i}$ is of the form
$$\Ind_{ (N_{\Ga,d,i})_{\theta_{\Ga,d,i}}\wr\prod_j
  \fS(t_{\Ga,d,i,j})}^{ (N_{\Ga,d,i})_{\theta_{\Ga,d,i}}\wr\fS(t_{\Ga,d,i})}
  \left(\zeta_{\Ga,d,i}\cdot\prod_j\phi_{\kappa_{\Ga,d,i,j}}\right),
$$
where $t_{\Ga,d,i}=\sum_j t_{\Ga,d,i,j}$, $\zeta_{\Ga,d,i}$ is an extension of
$\prod_j \psi_{\Ga,d,i,j}^{t_{\Ga,d,i,j}}$ from $((N_{\Ga,d,i})_{\theta_{\Ga,d,i}})^{t_{\Ga,d,i}}$ to $(N_{\Ga,d,i})_{\theta_{\Ga,d,i}}\wr
\prod_j\fS(t_{\Ga,d,i,j})$, $\kappa_{\Ga,d,i,j} \vdash t_{\Ga,d,i,j}$ is an $e_\Ga$-core and $\phi_{\kappa_{\Ga,d,i,j}}$ is a character of $\fS(t_{\Ga,d,i,j})$ corresponding to $\kappa_{\Ga,d,i,j}$.

From this, the $B$-weights are in bijection with assignments
$$\coprod_\Ga\coprod_{d\ge 0}\scC_{\Ga,d}\to\{ 2\text{-cores} \},
  \quad \psi_{\Ga,d,i}\mapsto \kappa_{\Ga,d,i},$$
such that
$\sum\limits_{d\ge 0}2^d\sum\limits_{i=1}^{|\scC_{\Ga,d}|}|\kappa_{\Ga,d,i}|
=w_\Ga$.
Using the $2$-core towers for partitions from \cite[(1A)]{AF90} (or
\cite[\S 6]{Ol93}) one can check that for $\Ga\in\cF_1\cup\cF_2$, the
assignments $\coprod\limits_{d\ge 0}\scC_{\Ga,d}\to\{ 2\text{-cores} \}$,
$\psi_{\Ga,d,i}\mapsto \kappa_{\Ga,d,i}$, such that
$\sum\limits_{d\ge 0}2^d\sum\limits_{i=1}^{2^d}|\kappa_{\Ga,d,i}|=w_\Ga$ are
in bijection with $\mathscr P(w_\Ga)$. If $\Ga=x-1$, then 
$$\sum_{i=1}^{2}\bigg(\sum_{d\ge 0}2^d\sum_{j=1}^{2^d}|\kappa_{x-1,d,i,j}|\bigg)+|\kappa_{x-1,d,3,1}|=w_{x-1}$$
and then the assignments are in bijection with $\scT(w_{x-1})$.
Thus, indeed,
$\scT(w_{x-1})\times\prod\limits_{\Ga\in \cF_1\cup\cF_2} \mathscr P(w_\Ga)$ is
a labelling set for $\cW(B)$.
\medskip

If $\mu\in\scT(m)$ corresponds to
$(\la_1,\la_2,\la_3,\kappa_1,\kappa_2)\in\scT'(m)$, then we define $\mu^\dag$
to be the element of $\scT(m)$ corresponding to
$(\la_1,\la_3,\la_2,\kappa_2,\kappa_1)$.

\begin{prop}   \label{act-wei}
 Let $(R,\vhi)$ be a $2$-weight of $G$ belonging to the block $B_s$ and
 corresponding to $\big(\mu_{x-1},\prod\limits_{\Ga\in \cF_1\cup\cF_2}\mu_\Ga\big)
 \in\scT(w_{x-1})\times\prod\limits_{\Ga\in \cF_1\cup\cF_2} \mathscr P(w_\Ga)$.
 Then:
 \begin{enumerate}[\rm(a)]
  \item For $\sigma\in\langle F_p\rangle$, $(R,\vhi)^{\sigma}$ is a $2$-weight
   of $G$ belonging to the block $B_{{\sigma^*}^{-1}(s)}$ and corresponding to
   $\big(\mu_{x-1},\prod\limits_{\Ga\in \cF_1\cup\cF_2} \mu_{{}^{{\sigma^*}^{-1}}\Ga}\big)$, where ${}^{{\sigma^*}^{-1}}\Ga$ is defined as in \cite[p.~398]{FLZ19a}.
  \item $(R,\vhi)^\delta$ is a $2$-weight of $G$ belonging to the block $B_s$
   and corresponding to $\big(\mu_{x-1}^\dag,\prod\limits_{\Ga\in \cF_1\cup\cF_2} \mu_\Ga\big)$.
 \end{enumerate}
\end{prop}

\begin{proof}
Using Lemma~\ref{act-wei-basic-case}, (\ref{conven-weights-basic-case-1})
and~(\ref{conven-weights-basic-case-2}), the proofs of Lemma~3.6 and
Proposition~3.11 of \cite{FLZ19a} also apply here.
\end{proof}

\begin{rem}
By the above result, the action of $\langle F_p\rangle$ on $\cW(G)$ is induced
by the action of $\langle F_p^*\rangle$ on elementary divisors of semisimple
elements of $G^*$.
\end{rem}

\subsection{A Jordan decomposition for weights}
\label{Jordan-decomposition-for-weights}

Is there an analogue of Jordan decomposition for weights of finite groups of
Lie type? This question is posed in Problem~4.9 of \cite{Ma15}.
We give a positive answer for $2$-weights of $G=\Sp_{2n}(q)$.

Let $s$ be a semisimple $2'$-element of $G^*$. Recall from (\ref{decom-L}) that
$$C_{G^*}(s)^*\cong \Sp_{2w_{x-1}}(q)\times \prod_{\Ga\in\cF_1\cup \cF_2} \GL_{m_\Ga(s)}(\veps_\Ga q^{\delta_\Ga}).$$

\begin{prop}   \label{Jor-dec-weights}
	Let $s$ be a semisimple $2'$-element of $G^*$ and $b_s$ the unipotent
	$2$-block of $C_{G^*}(s)^*$. Then there is a bijection $\mathscr J_s$ between 
	$\cW(b_s)$ and $\cW(B_s)$ such that 
	for $\sigma\in\langle F_p\rangle$ or $\sigma=\delta$ the following diagram commutes
	\begin{equation}\label{comm-diag}
	\xymatrix{
		\cW(b_s) \ar[r]^{\sigma}\ar[d]_{\mathscr J_s}& 	\cW(b_{s_\sigma}) \ar[d]_{\mathscr J_{s_\sigma}}\\
		\cW(B_s)\ar[r]^{\sigma}& 	\cW(B_{s_\sigma})}
	\addtocounter{thm}{1}\tag{\thethm}
	\end{equation}
	where $s_\sigma={\sigma^*}^{-1}(s)$ if $\sigma\in\langle F_p\rangle$, and
	$s_\sigma=s$ if $\sigma=\delta$ defined as in Proposition \ref{Jor-dec-Bar}. 
\end{prop}

\begin{proof}
By \cite{An92,An93b}, $\mathscr P(m)$ is a labelling set for the unipotent
$2$-blocks of the general linear and unitary group $\GL_m(\pm q)$. The action
of automorphisms on weights of $\GL_m(\pm q)$ is given in \cite[\S 5]{LZ18}:
the action of field or graph automorphisms on $2$-weights of $\GL_m(\pm q)$
is induced by the action of field or graph automorphisms on elementary
divisors. So field automorphisms fix all $2$-weights of unipotent $2$-blocks of
$\GL_m(\pm q)$.
Furthermore, for $C_{G^*}(s)^*$ as in~(\ref{decom-L}), every component
$\GL_{m_\Ga(s)}(\veps_\Ga q^{\delta_\Ga})$ ($\Ga\in\cF_1\cup\cF_2$) commutes
with $\langle\delta\rangle$.

The bijection $\mathscr J_s$ is now given in terms of the
combinatorial parametrisations for $\cW(b_s)$ and $\cW(B_s)$: a $2$-weight in
$\cW(b_s)$ corresponds to the $2$-weight in $\cW(B_s)$ with the same parameters.
The commutative diagrams can be checked directly by the above paragraph and
Proposition \ref{act-wei}.
\end{proof}

We have further positive answers to \cite[Problem~4.9]{Ma15}. Let
\begin{itemize}
 \item $G=\GL_n(\pm q)$ and $\ell$ any prime not dividing $q$; or
 \item $G=\Sp_{2n}(q)$ or $\SO_{2n+1}(q)$ with odd $q$ and $\ell$ an odd
  prime not dividing $q$.
\end{itemize}
Let $s$ be a semisimple $\ell'$-element of $G^*$. 
By the combinatorial parametrisation of $\ell$-blocks of $G=\GL_n(\pm q)$ given
in \cite{FS82,Br86}, of $G=\SO_{2n+1}(q)$ given in \cite[(10B)]{FS89}, and of
$G=\Sp_{2n}(q)$ given in \cite[\S 5.2]{FLZ19a} there is a bijection between the
$\ell$-blocks in $\cE_\ell(G,s)$ and the unipotent $\ell$-blocks of
$C=C_{G^*}(s)$, and $C=C^\circ$ unless $G=\Sp_{2n}(q)$ and $-1$ is an
eigenvalue of $s$, in which case $|C/C^\circ|=2$. Let $B_s$ be an $\ell$-block
in $\cE_\ell(G,s)$ and let $b_s$ correspond to $B_s$ under this bijection.

Let $\sigma$ be a field or diagonal automorphism for $G=\Sp_{2n}(q)$, a field
automorphism for $G=\SO_{2n+1}(q)$, or a field or graph automorphism for
$G=\GL_n(\pm q)$. As before, we also let $s_\sigma={\sigma^*}^{-1}(s)$ if
$\sigma$ is a field or graph automorphism and $s_\sigma=s$ if $\sigma$ is a
diagonal automorphism here.

\begin{prop}   \label{Jor-dec-wei-GLGUSpSO}
 Keep the above hypotheses and setup. Then there exists a bijection
 $\mathscr J_s$ between $\cW(b_s)$ and $\cW(B_s)$ and (\ref{comm-diag})
 commutes for $G$, $\ell$, $B_s$, $b_s$ and $\sigma$ as given above.
\end{prop}

Here, the action of $\delta$ on the conjugacy classes of weights in unipotent
blocks is realised by the multiplication with the non-trivial linear character
of $C/C^\circ$ on $\cW(C)$.

\begin{proof}
We first introduce some notation from \cite{FS82,FS89}. For $G=\GL_n(\eps q)$
with $\eps =\pm 1$, we let $\cF$ be the set of polynomials defined in
\cite[p.~112]{FS82}, so that elements of $\cF$ serve as the ``elementary
divisors” for elements of $G$. Note that this $\cF$ for linear and unitary
groups here is not the same as the $\cF$ for symplectic and orthogonal groups
defined in \S \ref{notations-and-conventions}. For $\Ga\in\cF$, we let $d_\Ga$
be the degree of $\Ga$ and $e_\Ga$ the multiplicative order of
$(\eps q)^{d_\Ga}$ modulo $\ell$. Note that $e_\Ga=1$ for all $\Ga$ if $\ell=2$.
Then for a semisimple element $s\in G^*$ we denote by $m_\Gamma(s)$ the
multiplicity of $\Ga\in\cF$ in its primary decomposition.
By \cite{FS82,Br86}, the $\ell$-blocks of $G=\GL_n(\eps q)$ are in bijection
with the $G$-conjugacy classes of pairs $(s,\ka)$, where $s$ is a semisimple
$\ell'$-element of $G$, $\ka=\prod_\Ga \ka_\Ga$ and $\ka_\Ga$ is an
$e_\Ga$-core of some partition of $m_\Ga(s)$ for each $\Ga$. 
We write $C=\prod_\Ga C_\Ga$, where
$C_\Ga=\GL_{m_\Ga(s)}((\eps q)^{d_\Ga})$.
We let $w_\Ga$ be the integer such that $m_\Ga(s)=|\ka_\Gamma|+e_\Ga w_\Ga$.

For $G=\Sp_{2n}(q)$ or $\SO_{2n+1}(q)$ and an odd prime $\ell$, we use the
notation defined as before, especially in \S \ref{notations-and-conventions}.
Let $e_\Ga$ be the multiplicative order of $q^2$ or $\veps q^{\delta_\Ga}$
modulo $\ell$ according as $\Ga\in\cF_0$ or $\cF_1\cup\cF_2$.
By \cite[(10B)]{FS89} and \cite[\S 5.2]{FLZ19a}, the $\ell$-blocks of $G$ are in
bijection with the $G^*$-conjugacy
classes of pairs $(s,\ka)$, where $s\in G^*$ is a semisimple $\ell'$-element,
$\ka=\prod_\Ga \ka_\Ga$, $\ka_\Ga$ is an $e_\Ga$-core of some partition of
$m_\Gamma(s)$ if $\Ga\in\cF_1\cup\cF_2$, and $\ka_\Ga$ is an $e_\Ga$-core of
some Lusztig symbol~(see \cite[\S 2.4]{FLZ19a} for details) if $\Ga\in\cF_0$
and is counted twice if $G=\Sp_{2n}(q)$, $\Ga=x+1$ and $\ka_{x+1}$ is
non-degenerate (cf. \cite[Thm.~5.10]{FLZ19a}).
In particular, the unipotent $\ell$-blocks are labelled by certain
$e_\Ga$-cores. By \cite[Thm.]{CE94}, the unipotent $\ell$-blocks of
$\SO_{2n}^\pm(q)$ (and then of $\GO_{2n}^\pm(q)$) are also labelled by 
$e_\Ga$-cores; see e.g. \cite[\S 6]{FLZ19a}. We write  $C=\prod_\Ga C_\Ga$,
where $C_\Ga=\GL_{m_\Ga(s)}(\varepsilon_\Ga q^{\delta_\Ga})$ if
$\Ga\in\cF_1\cup\cF_2$, $C_\Ga=\SO_{m_\Ga(s)+1}(q)$ if $G=\SO_{2n+1}(q)$ and
$\Ga\in\cF_0$, and $C_{x-1}=\Sp_{m_{x-1}(s)-1}(q)$ and
$C_{x+1}=\GO^\pm_{m_{x+1}(s)}(q)$ for $G=\Sp_{2n}(q)$. Now we let $w_\Ga$ be
the integer
such that $m_\Ga(s)=|\ka_\Gamma|+e_\Ga w_\Ga$ if $\Ga\in\cF_1\cup \cF_2$, and
$\lfloor\frac{m_\Ga(s)-1}{2}\rfloor=|\ka_\Gamma|+e_\Ga w_\Ga$ if $\Ga\in\cF_0$.

Then in all cases, if $B_s$ is an $\ell$-block with label $(s,\ka)$, then
$b_s=\otimes_\Ga b_\Ga$, where $b_\Ga$ is a unipotent $\ell$-block of $C_\Ga$
with label $\ka_\Ga$.

If $d\ge 1$ and $m$ are integers, then we denote by $\mathscr P(d,m)$ the set
of tuples $(\lambda_1,\ldots,\lambda_d)$ of partitions such that
$|\lambda_1|+\cdots+|\lambda_d|=m$. Then by \cite{AF90,An92,An93b,An94},
$\prod_\Ga \mathscr P_\Ga$ is a labelling set for the $\ell$-weights
$\cW(B_s)$, where $\mathscr P_\Ga=\mathscr P(e_\Ga,w_\Ga)$ unless
$G=\Sp_{2n}(q)$ or $\SO_{2n+1}(q)$ and $\Ga\in\cF_0$, in which case
$\mathscr P_\Ga=\mathscr P(2e_\Ga,w_\Ga)$. On the other hand, we can check
directly that $\mathscr P_\Ga$ is a labelling set for $\cW(b_\Ga)$. 
The bijection $\mathscr J_s$ can now be defined in terms of these
parametrisations for $\cW(b_s)$ and $\cW(B_s)$: an $\ell$-weight in $\cW(b_s)$
corresponds to the $\ell$-weight in $\cW(B_s)$ with the same parameters.
 
By \cite[Thm.~1.1]{LZ18}, the actions of field and graph automorphisms on
weights of $\GL_n(\eps q)$ are induced by the actions on elementary divisors.
By \cite[Prop.~4.7 \& 5.14]{FLZ19a}, the action of field automorphisms on
weights of $\Sp_{2n}(q)$ or $\SO_{2n+1}(q)$ is also induced by the action on
elementary divisors. In addition, by \cite[Prop.~5.14]{FLZ19a} again, the
diagonal automorphism $\delta$ fixes all weights of $B_s$ unless
$G=\Sp_{2n}(q)$ and $-1$ is an eigenvalue of $s$, in which case $\delta$ acts
fixed-point freely on $\cW(\cE_\ell(G,s))$.
In particular, if $\sigma$ is a graph automorphisms of $\GL_n(\eps q)$ or a
field automorphism, then $\sigma$ fixes all weights of $b_s$ (for the assertion
that field automorphisms fix the weights of the unipotent block $b_{x+1}$ of
$C_{x+1}\cong \GO^\pm_{m_{x+1}(s)}(q)$, see \cite[Lemma 6.3]{FLZ19a}).
If $G=\Sp_{2n}(q)$, then the weights of $b_\Ga$ are invariant under $\delta$
unless $\Ga= x+1$, in which case $\delta$ acts by multiplication with the
non-trivial linear character of $C_{x+1}$ on $\cW(\cE_\ell(C_{x+1},1))$. Thus
the weights in $\cW(\cE_\ell(C_{x+1},1))$ are also interchanged by $\delta$.
From this the commutative diagrams can be checked directly for all cases.
\end{proof}

\subsection{Weights of the unipotent $2$-block}

Now we focus on the unipotent (principal) $2$-block $B_1$ of $G=\Sp_{2n}(q)$
and summarise the results for the action of $\Aut(G)$ on $\cW(B_1)$.
Both $\scT(n)$ and $\scT'(n)$ are labelling sets for the $B_1$-weights of $G$.
Let $(R,\vhi)$ be a $B_1$-weight corresponding to
$(\la_1,\la_2,\la_3,\kappa_1,\kappa_2)\in \scT'(n)$.
Then by Proposition~\ref{act-wei},
$(R,\vhi)^\delta$ corresponds to $(\la_1,\la_3,\la_2,\kappa_2,\kappa_1)$.

Let $\scT'_1(n)$ be the subset of $\scT'(n)$ consisting of tuples
$(\la_1,\la_2,\la_3,\kappa_1,\kappa_2)$ such that $\la_2=\la_3$ and
$\kappa_1=\kappa_2$. Then by Proposition \ref{act-wei} we have:

\begin{prop}   \label{act-wei-uni}
 \begin{enumerate}[\rm(a)]
  \item Every $B_1$-weight of $G$ is invariant under the action of field
   automorphisms.
  \item If $(R,\vhi)$ is a $B_1$-weight, then the $G$-conjugacy class of
   $2$-weights of $G$ containing $(R,\vhi)$ is $\tG$-invariant if and only if
   $(R,\vhi)$ corresponds to an element in $\scT'_1(n)$.
 \end{enumerate}
\end{prop}

\section{An equivariant bijection}   \label{sec:equ bijection uni}

Recall that $B_1$ is the unipotent $2$-block of $G=\Sp_{2n}(q)$ with $q$ odd.

\begin{prop}   \label{AWC-B0}
 We have $|\IBr(B_1)|=|\cW(B_1)|$.
\end{prop}

\begin{proof}
It suffices to show that $|\scU(n)|=|\scT(n)|$.
First by \cite[p.38]{Wa63}, $|\scU(n)|$ is the coefficient of $t^n$ in
$$\prod_{k=1}^\infty\frac{(1+t^k)^2}{1-t^k}
  =\prod_{k=1}^\infty \frac{(1+t^k)(1-t^{2k})}{(1-t^k)^2}.$$

On the other hand, $|\scT(n)|$ is the coefficient of $t^n$ in
\begin{align*}
\left( \sum_{k=0}^\infty t^k \right )\left(\sum_{k=0}^\infty t^{2k}\right)
\left(\sum_{k=0}^\infty t^{3k}\right)\cdots\left(\sum_{k=0}^\infty t^k\right)&\left(\sum_{k=0}^\infty t^{2k}\right)
\left(\sum_{k=0}^\infty t^{3k}\right)\cdots 
\left(\sum_{k=0}^\infty t^{\frac{k(k+1)}{2}}\right)\\
&=\left(\prod_{k=1}^\infty\frac{1}{ (1-t^k)^2}\right)\left(\sum_{k=0}^\infty t^{\frac{k(k+1)}{2}}\right)
\end{align*}
In order to prove $|\scU(n)|=|\scT(n)|$, it suffices to show that
$$\sum_{k=0}^\infty t^{\frac{k(k+1)}{2}}=\prod_{k=1}^\infty (1+t^k)(1-t^{2k}),$$
and this follows directly by Jacobi's identity (see, e.g.,
\cite[Thm.~354]{HW08}). This completes the proof.
\end{proof}

\begin{thm}   \label{thm:equibij}
 There is an $\Aut(G)$-equivariant bijection between $\IBr(B_1)$ and $\cW(B_1)$.
 \end{thm}

\begin{proof}
By Corollary~\ref{cor:aut IBR} and Proposition~\ref{act-wei-uni}, field
automorphisms fix all elements in $\IBr(B_1)\cup\cW(B_1)$. As regards the
action of diagonal automorphisms, we note that an element of $\IBr(B_1)$
(resp.~$\cW(B_1)$) is either $\tG$-invariant or it lies in an orbit of
length~2. Thanks to Proposition \ref{AWC-B0},
it suffices to prove that the $\tG$-fixed elements of $\IBr(B_1)$ and
$\cW(B_1)$ have equal cardinality, and then by Corollary~\ref{cor:aut IBR} and
Proposition~\ref{act-wei-uni} again it suffices to show that $|\scU_1(n)|=|\scT'_1(n)|$.

First, $|\scT'_1(n)|$ is the number of triples $(\mu,\la,\kappa)$, where 
$\mu$ and $\la$ are partitions, and $\kappa$ is a $2$-core such that
$|\mu|+4|\la|+2|\kappa|=n$. Thus $|\scT'_1(n)|$ is the coefficient of $t^n$ in
\begin{align*}
\left(\sum_{k=0}^\infty t^k\right)\left(\sum_{k=0}^\infty t^{2k}\right)
\left(\sum_{k=0}^\infty t^{3k}\right)\cdots\left(\sum_{k=0}^\infty t^{4k}\right)&\left(\sum_{k=0}^\infty t^{8k}\right)
\left(\sum_{k=0}^\infty t^{12k}\right)\cdots 
\left(\sum_{k=0}^\infty t^{k(k+1)}\right)\\
&=\left(\prod_{k=1}^\infty\frac{1}{ 1-t^k}\right)\left(\prod_{k=1}^\infty\frac{1}{ 1-t^{4k}}\right)\left(\sum_{k=0}^\infty t^{k(k+1)}\right).
\end{align*}

On the other hand, $|\scU_1(n)|$ is the number of maps $\bm':\ZZ_{\ge 1}\to \ZZ_{\ge 0}$ such that $\sum_{j\ge 1} j\bm'(j)=n$.
Here $\bm'$ is counted $2^{k_{\bm'}}$-times, where 
$k_{\bm'}=|\{ j\mid j\ \text{is even and}\ \bm'(j)\ne 0 \}|$.
Hence $|\scU_1(n)|$ is the coefficient of $t^n$ in
\begin{align*}
\left(\sum_{k=0}^\infty t^k\right)\left(1+2\sum_{k=1}^\infty t^{2k}\right)
\left(\sum_{k=0}^\infty t^{3k}\right)&\left(1+2\sum_{k=1}^\infty t^{4k}\right)\left(\sum_{k=0}^\infty t^{5k}\right)\cdots\\
&=\left(\frac{1}{1-t}\right)\left(\frac{1+t^2}{1-t^2}\right)\left(\frac{1}{1-t^3}\right)\left(\frac{1+t^4}{1-t^4}\right)\left(\frac{1}{1-t^5}\right)\cdots\\
&=\left(\prod_{k=1}^\infty\frac{1}{ 1-t^k}\right)\left(\prod_{k=1}^\infty (1+t^{2k})\right).
\end{align*}

Thus in order to prove $|\scU_1(n)|=|\scT'_1(n)|$, it suffices to show
that
$$\sum_{k=0}^\infty t^{k(k+1)}=\prod_{k=1}^\infty (1+t^{2k})(1-t^{4k}),$$
and this again follows by Jacobi's identity (cf. \cite[Thm.~354]{HW08}).
This completes the proof.
\end{proof}

\vspace{2ex}

Now we are ready to prove our main result.

\begin{proof}[Proof of Theorem \ref{mainthm}]
Thanks to \cite[Thm.~C]{Sp13}, we only need to consider non-defining
characteristic and thus may assume that $q$ is odd. By
Proposition~\ref{ibaw-condition-C-2}, Remark~ \ref{equiva-condition(ii)(iii)}
and Corollary~\ref{cor:(3)}, the inductive blockwise Alperin weight condition
holds for the simple group $S=\PSp_{2n}(q)$ and the prime $2$ if there is an
$\Aut(G)$-equivariant bijection between $\IBr(G)$ and $\cW(G)$ which preserves
$2$-blocks.   \par
According to Propositions~\ref{Jor-dec-Bar} and~\ref{Jor-dec-weights}, we only
need to consider the unipotent $2$-block of $C_{G^*}(s)^*$. Note that
$C_{G^*}(s)^*$ is a product of a symplectic group with some general linear and
unitary groups and the blockwise bijections between the irreducible $2$-Brauer
characters and conjugacy classes of $2$-weights of general linear and unitary
groups constructed by  An 
are equivariant under the
action of field and graph automorphisms by \cite[Thm.~1.1]{LZ18}. Thus, we
only need to consider the unipotent (principal) $2$-block of the symplectic
group $G$, which is dealt with in Theorem~\ref{thm:equibij}.
\end{proof}


\end{document}